\documentclass[11pt]{amsart}
\usepackage{graphicx}
\usepackage{amssymb}

\newtheorem{thm}{Theorem}[section]
\newtheorem{cor}[thm]{Corollary}
\newtheorem{lem}[thm]{Lemma}
\newtheorem{prop}[thm]{Proposition}
\theoremstyle{definition}
\newtheorem{defn}[thm]{Definition}
\theoremstyle{example}
\newtheorem{exam}[thm]{Example}
\theoremstyle{remark}

\numberwithin{equation}{section}

\theoremstyle{definition}

\theoremstyle{remark}

\numberwithin{equation}{section}
\newcommand{\h}{\mathcal{H}}
\newcommand{\NN}{\mathbb N}

\newcommand{\RR}{\mathbb R}
\newcommand{\CC}{\mathbb C}
\newcommand{\SSS}{\mathbb S}
\begin{document}

\title{Redundancy of Fusion frames in Hilbert Spaces}

\author[ A. Rahimi, G. Zandi and B. Daraby]{ A. Rahimi, G. Zandi and B. Daraby}
\address{ Department of Mathematics, University of Maragheh, P. O. Box 55181-83111, Maragheh, Iran.}

\email{rahimi@maragheh.ac.ir}
\address{ Department of Mathematics, University of Maragheh, P. O. Box 55181-83111, Maragheh, Iran.}
\email{zgolaleh@yahoo.com}
\address{ Department of Mathematics, University of Maragheh, P. O. Box 55181-83111, Maragheh, Iran.}
\email{bdaraby@maragheh.ac.ir}

\subjclass[2000]{Primary 42C40; Secondary 41A58, 47A58.}
\keywords{Fusion Frame, Redundancy Function, Upper Redundancy, Lower Redundancy, Erasure.}
\begin{abstract}

Upon improving and extending the concept of redundancy of frames,
 we introduce the notion of redundancy of
fusion frames, which is concerned with the properties of lower and upper redundancies.   These properties are achieved by
considering the  minimum and maximum values of the redundancy function which is defined from the unit sphere of the Hilbert space into the positive real numbers.
 In
addition, we study the relationship between redundancy of frames (fusion frames) and dual frames (dual fusion frames). Moreover, we indicate some results about excess of fusion frames. We state the relationship between redundancy of local frames and fusion frames in a particular case. Furthermore, some examples are also given.\end{abstract}

\maketitle

\section{Introduction}

Frames for Hilbert spaces have been introduced in $1952$ by Duffin and Schaeffer in their fundamental paper \cite{duffin}
and have been studied in the last two decades
as a powerful framework for robust and stable representation
of signals by introducing redundancy. The customary definition of redundancy was improved by Bodmann, Casazza and Kutyniok in \cite{2} by providing a quantitative measure, which coined upper and lower redundancies.

Redundancy is applied in areas such as: filter bank theory \cite{3'} by Bolcskei, Hlawatsch and Feichtinger, sigma-delta quantization \cite{1'} by Benedetto, Powell and Yilmaz,
signal and image processing \cite{5'} by Cand\`{e}s and Donoho and wireless communications \cite{121} by Heath and Paulraj.
However, many of  the applications can not be modeled by one single frame system.
They require
   distributed processing such as sensor networks \cite{12'}. To handle some emerging applications of frames, new methods developed.
   One starting point was to first build frames ``locally'' and then piece
   them together to obtain frames for the whole space.
  So we can first construct frames  or choose already known frames for smaller spaces,
   and in the second step one would construct a frame for the whole space from them.
   Another construction uses subspaces which are quasi-orthogonal to construct local frames
   and piece them together to get global frames\cite{for}. An elegant approach was introduced in
   \cite{6} that formulates a general method for piecing together local frames to get global frames. This powerful construction was
    introduced by  Casazza and Kutyniok in \cite{6}, named frames of subspaces which
   thereafter they agree on a terminology of fusion frames. This
 notion provides a useful framework
in modeling sensor networks \cite{9'}.

Fusion frames can be regarded as a generalization of conventional frame theory.
It turns out that the fusion frame theory is in fact more
 delicate due to complicated
relationships between the structure of the sequence of weighted
subspaces and the local frames in the subspaces and due to
sensitivity with respect to change of the weights. Redundancy is a crucial
property of a fusion frame as well as of a frame. In the situation
of frames, the rather crude measure of the number of frame vectors
divided by the dimension (in the finite dimensional case) is defined
as the redundancy  which is the
frame bound in the case of tight frame with normalized vectors. This concept has been
 replaced by a more appropriate measure (see \cite{2}).

In this paper, we will focus on the  study of  redundancy of the fusion frames. Furthermore, we will state the relationship between redundancy of local frames and fusion frames in a special case.\par
 At the first, we will review the basic definitions related to the fusion frames.
Throughout this paper, $\h$ is a real or complex Hilbert space and ${\h}^{n}$ is an n-dimensional Hilbert space.

\begin{defn}
Let $\h$ be a  Hilbert space and $I$ be a (finite or infinite) countable index set. Assume that $\{W_i\}_{i\in
I}$ be a sequence of closed subspaces in $\h$ and $\{v_i\}_{i\in I}$ be a family
of weights, i.e., $v_i>0$ for all $i\in I$. We say that the
family $\mathcal{W}=\{(W_i,v_i)\}_{i\in I}$ is a  $fusion$ $frame$ or a $frame$ $of$ $subspaces$ with
respect to $\{v_i\}_{i\in I}$ for $\h$ if there
exist constants $0<A\leq B<\infty$ such that
$$A\|x\|^2\leq\sum_{i\in I}v_i^2\|P_{W_i}(x)\|^2\leq
B\|x\|^2\quad\forall x\in\h,$$ where $P_{W_i}$ denotes the orthogonal
projection onto $W_i,$ for each $i\in I.$ The fusion frame
$\mathcal{W}=\{(W_i,v_i)\}_{i\in I}$ is called  $tight$  if
$A=B$ and $Parseval$ if $A=B=1$.  If all ${v_i}^{,}s$ take the same value $v$, then $\mathcal{W}$ is
called $v$-$uniform$.
Moreover, $\mathcal{W}$ is called an $orthonormal$ $fusion$
 $basis$ for $\h$ if  $\h= \bigoplus_{i\in I} W_i$.
 If $\mathcal{W}=\{(W_i,v_i)\}_{i\in I}$ possesses an
 upper fusion frame bound but not necessarily a lower
 bound, we call it a $Bessel$ $fusion$ $sequence$ with Bessel fusion bound $B$.
 The normalized version of $\mathcal{W}$ is obtained when we choose $v_i=1$ for all $i\in I.$
 Note that we use this term merely when $\{(W_i,1)\}_{i\in I}$ formes a fusion frame for $\h.$
\end{defn}
Without loss of generality,  we may assume that the family
of weights $\{v_i\}_{i\in I}$ belongs to $\ell^{\infty}_+(I) $.\\
\indent As we know, redundancy appears as a mathematical concept and as a methodology
for signal processing. Recently, the ability of redundant systems to provide
sparse representations has been extensively exploited \cite{3}. In fact, frame
 theory is entirely based on the notion of redundancy.

As a first notion of redundancy in the situation of a tight fusion
frame, we can choose its fusion frame bound as a measure which is
equivalent to $$A=\sum_{i=1}^{M}{ v^2_i \dim W_i \over n},$$ where
$\{(W_i,v_i)\}_{i=1}^{M}$  is an A-tight fusion frame for a Hilbert
space $\mathcal{H}^n$  \cite{7}. We will illustrate that this measurement of redundancy is
applied only for tight fusion frames by introducing and analysing two examples of fusion
frames. Let $\{e_i\}_{i=1}^{n}$ be an orthonormal
basis for a Hilbert space $\mathcal{H}^n$ and so a normalized
Parseval frame for $\mathcal{H}^{n}$. Let $W_i=\operatorname {span}\{e_i\}$
and $v_i=$1=$\|e_i\|$, for each
$i=1,...,n$. Then
   $\{(W_i,v_i)\}_{i=1}^{n}$ is a Parseval fusion frame
for $\mathcal{H}^{n}$. Now consider $$\mathcal{W}=
\{W_1,...,W_1,W_2,...,W_n\},$$ where  $W_1$ occurs $n+1$ times
and$$\mathcal{V}=\{W_1,W_1,W_2,W_2,...,W_n,W_n\}.$$  It is obvious
that $\mathcal{W}$ and  $\mathcal{V}$ are fusion frames for
$\mathcal{H}^{n}$ with respect to $v_i=1$, for each $i$. For
$\mathcal{W}$ and $\mathcal{V}$, the above measure of redundancy
coincides and is equal to$$\sum_{i=1}^{2n}{ v^2_i \dim W_i \over {n}
}={2n\over n}=2.$$ However, intuitively the redundancy of
$\mathcal{W}$ seems to be  localized, while the redundancy of
$\mathcal{V}$ seems to be uniform. The fusion frame $\mathcal{V}$
is robust with respect to any one erasure, whereas $\mathcal{W}$
does not have this property. Neither of these facts can be read from
the above redundancy notion, which is not good enough. Ideally,
the upper redundancy of $\mathcal{W}$  should be
$n+1$ and the lower 1, while the upper and lower redundancies
of $\mathcal{V}$ should coincide and equal $2$.
 More generally, if a fusion frame consists of an orthonormal fusion
  basis which is individually repeated several times, then the
   lower  redundancy should be the smallest number of repetitions
    and the upper redundancy the largest.\\
\indent In order to define a better notion of redundancy for fusion frames,
 we will consider a list of properties that our notion is required to satisfy, similar to those in \cite{2} and \cite{cahill}.\\
$\textbf{Outline.}$ The outline is as follows:\\
We will start our consideration by giving a brief review of the definitions and basic
properties of fusion frames   in Section 2.
We will define the redundancy function for finite (infinite) fusion frames and state main results in Section 3.
 The relationships between redundancy of frames (fusion frames) and
dual frames (dual fusion frames) will be investigated in Section 4.
The concept of excess of fusion frames will be reviewed and discussed in Section 5.
 In Section 6, the relationship between redundancy of fusion frames
  and local frames are discussed in a particular case.
 Finally, Section 7 contains some examples.

\section{Fusion Frames}
In this section, we will  review the definitions of the analysis, synthesis  and fusion
 frame operator introduced in \cite{6}. Moreover we will
  state some already proved properties and theorems around fusion frames.\\ \par
 \textbf{Notation:} For any family $\{\h_i\}_{i\in I}$ of Hilbert spaces, we use
 \[(\sum_{i\in I}{\oplus \h_i})_{\ell_2}= \left\{ \{f_{i}\}_{i\in I}: f_i\in\h_i,  \sum_{i\in I}\|f_{i}\|^2<\infty
\right\}\]
with inner product
\[
\langle\{f_i\}_{i\in I},\{g_i\}_{i\in I}\rangle=\sum_{i\in I} \langle f_i, g_i\rangle,\quad \{f_{i}\}_{i\in I}, \{g_{i}\}_{i\in I}\in(\sum_{i\in I}{\oplus \h_i})_{\ell_2}
\]
and
\[
\|\{f_i\}_{i\in I}\|:=\sqrt{\sum_{i\in I} \|f_i\|^2}.
\] It is easy to show that $(\sum_{i\in I}{\oplus \h_i})_{\ell_2}$ is a Hilbert space.

 \begin{defn}
 Let $\mathcal{W}=\{(W_i,v_i)\}_{i\in I}$ be a fusion frame for $\h$.
 The $synthesis$ $operator$ $T_{\mathcal{W}}:(\sum_{i\in I}{\oplus W_i})_{\ell_2} \to \h$ is defined by
  $$  T_{\mathcal{W}}(\{f_i\}_{i\in I})= \sum_{i\in I}{v_if_i}, \quad \{f_i\}_{i\in I}\in (\sum_{i\in I}{\oplus W_i})_{\ell_2}.$$
 In order to map a signal to the representation space, i.e., to analyze it,
  the $analysis$ $operator$ $T_{\mathcal{W}}^{*}$ is employed which is defined by
  $$T_{\mathcal{W}}^{*}: \h \to (\sum_{i\in I}{\oplus W_i})_{\ell_2} \quad with \quad
  T_{\mathcal{W}}^{*}(f)=\{v_i P_{{W_i}}(f)\}_{i\in I},$$ for any $f\in\h$. The $fusion$ $frame$ $operator$ $S_{\mathcal{W}}$ for $\mathcal{W}$
  is defined by
  $$S_{\mathcal{W}}(f)=T_{\mathcal{W}}T_{\mathcal{W}}^{*}(f)=\sum_{i\in I} v^{2}_i P_{W_i}(f),\quad f\in\h.$$
 \end{defn}
 It follows from \cite{6} that for each fusion frame, the operator $S_{\mathcal{W}}$  is invertible,  positive and $AI\leq S_{\mathcal{W}}\leq BI$. Any $f\in\h$ has the representation $f=\sum_{i\in I} v^{2}_i S_{\mathcal{W}}^{-1}P_{W_i}(f)$.\\
\indent Let us state  some definitions  and propositions that we need in this paper.

\begin{prop}\cite{6}\label{111}
 Let $\{ W_i\}_{i\in I}$ be a family of subspaces for $\h$. Then the following conditions are equivalent.
 \begin{enumerate}
\item $\{ W_i\}_{i\in I}$ is an orthonormal fusion basis for $\h$;
 \item $\{W_i\}_{i\in I}$ is a 1-uniform Parseval fusion frame
for $\h$.\\
\end{enumerate}
\end{prop}

\begin{defn}\cite{6}
We call a fusion frame $\mathcal{W}=\{(W_i, v_i)\}_{i\in I}$ for $\mathcal{H}$ a
\emph{Riesz decomposition} of $\mathcal{H},$  if  every $f\in
\mathcal{H}$  has a unique representation
$f=\sum_{i\in I} f_i,$ $f_i\in W_i.$
\end{defn}
 \begin{prop}\cite{6}
If $\mathcal{W}=\{(W_i, v_i)\}_{i\in I}$ is an orthonormal fusion basis for $\mathcal{\h}$, then it is also
 a Riesz decomposition of $\mathcal{\h}$.\\
 \end{prop}

 \begin{prop}\label{8}\cite{8}
 Let $\mathcal{W}=\{(W_i, v_i)\}_{i\in I}$ be a fusion frame with bounds $A$ and $B$ and let $J\subset I$.
 Then $\{W_i\}_{i\in I\backslash J}$ is a fusion frame with bounds $A-a$ and $B$ if $a=\sum_{i\in J}{v^{2}_i}<A$.\\
 \end{prop}

\begin{prop}\cite{7}\label{aaa}
Let $\mathcal{W}=\{(W_i, v_i)\}_{i\in I}$ be a fusion frame with
bounds $A$ and $B$. If $U$ is an invertible operator on $\mathcal{H}$, then $\{(UW_i,
v_i)\}_{i\in I}$ is a fusion frame for $\mathcal{H}$ with fusion
frame bounds $${A\over {\|U\|^{2}\|U^{-1}\|^{2}}} \quad and \quad
B\|U\|^{2}\|U^{-1}\|^{2}.$$
 \end{prop}

\begin{defn}\cite{6}
A family of subspaces $\{W_i\}_{i\in I}$ of $\h$ is called $minimal$ if for each $i\in I,$
$$W_i\cap \ \overline{\operatorname{span}}_{j\neq i}\{W_j\}_{j\in I}=\{0\}.$$
\end{defn}
\begin{prop}\cite{6}\label{bbb}
Let $\mathcal{W}=\{(W_i,v_i)\}_{i\in I}$ be a fusion frame for $\h.$
Then the following conditions are equivalent.
\begin{enumerate}
\item $\mathcal{W}=\{(W_i,v_i)\}_{i\in I}$ is a Riesz decomposition of $\h$;
\item $\{W_i\}_{i\in I}$ is minimal;
\item the synthesis operator is one to one;
\item the analysis operator is onto.
\end{enumerate}
\end{prop}
The following lemma is easy to prove.

\begin{lem}\label{2.9}
Let $I$ and $J$ be countable index sets and  $\mathcal{W}=\{(W_i,w_i)\}_{i\in I}$ and
  $\mathcal{V}=\{({V_j},{v_j})\}_{j\in J}$ be two fusion frames for $\h$.
  Then $\{(W_i,w_i)\}_{i\in I}\cup \{({W_j},{v_j})\}_{j\in J}$ is also a fusion frame for
   $\h.$
\end{lem}

\section{ Redundancy of Fusion Frames and the Main Result}
In this section, we present the definition of redundancy for fusion frames. A quantitative notion of redundancy
of finite frames and infinite frames was introduced  in \cite{2} and \cite{cahill}.
Our approach is similar to these works  for generalizing the concept of frame redundancy to fusion frames.
\subsection{Redundancy of Finite Fusion Frames}
By getting some ideas from the concept of the redundancy of finite frames in Hilbert spaces
 and lower and  upper redundancies \cite{2}, we define the
  redundancy function for finite fusion frames and we introduce and prove some of
  its properties. The redundancy function is defined from the unit sphere $\SSS=\{x\in{\h}^{n}:
\|x\|=1\}$ to the set of positive real numbers $\RR^+$.
\begin{defn}
Let $\mathcal{W}= \{(W_i, v_i)\}_{i=1}^{N}$ be a fusion frame for ${\h}^{n}$
with bounds $A$ and $B.$
For each $x\in \SSS$, the $redundancy$ $function$
$\mathcal{R}_{\mathcal{W}}:\SSS\to \RR^+$ is defined by $$
\mathcal{R}_{\mathcal{W}}(x)=\sum_{i=1}^{N}
\|P_{W_i}(x)\|^2
.$$

\end{defn}
Notice that this notion is reminiscent of the definition of redundancy function for finite frames  \cite{2}, if
$\dim W_i=1$ for $i=1,...,N.$
\begin{defn}
For the fusion frame $\mathcal{W}= \{(W_i, v_i)\}_{i=1}^{N},$ the \emph{upper redundancy}  is defined by
$$\mathcal{R}_{\mathcal{W}}^+=\sup_{x\in\SSS}
\mathcal{R}_{\mathcal{W}}(x), $$ and the \emph{lower redundancy} of
$\mathcal{W}$  by $$\mathcal{R}_{\mathcal{W}}^-=\inf_{x\in \SSS}
\mathcal{R}_{\mathcal{W}}(x). $$ We say that $\mathcal{W}$ has $uniform$
redundancy if $ \mathcal{R}_{\mathcal{W}}^-=
\mathcal{R}_{\mathcal{W}}^+$.
\end{defn}
This notion of redundancy equals the lower and  upper fusion frame bounds
of the normalized version of the fusion frame $\mathcal{W}= \{(W_i, v_i)\}_{i=1}^{N}$, i.e.,
when we take all ${v_i}^{,}s$ equal to 1.
By the definition of the redundancy function  it is obvious that $0<\mathcal{R}_{\mathcal{W}}^-\leq
\mathcal{R}_{\mathcal{W}}^+<\infty.  $ The convergence of the series and boundedness of orthogonal
projection imply the continuity of the redundancy function.\\
By using some linear algebra's concepts and tools, indeed the redundancy function is the \textit{Rayleigh quotient} of the fusion frame operator with respect to the normalized version of the fusion frame $\mathcal{W}= \{(W_i, v_i)\}_{i=1}^{N}.$
Recall from \cite{parlet} that for a Hermitian matrix $M$ and a nonzero vector $x,$ the Rayleigh quotient $R(M,x)$ is defined as follows:

$$R(M,x)={
 {{\langle x,Mx\rangle}\over {{\|x\|}^{2}}}} .$$
 Note that
$$R(M,x)=\langle {x\over\|x\|},{Mx\over\|x\|}\rangle=\langle u,Mu\rangle,\quad u={x\over\|x\|}.$$
So in fact, it is sufficient  to define the Rayleigh quotient on unit vectors. Hence by getting $M=S_{1\mathcal{W}}$
and $x\in \SSS$ we have 
$$R(S_{1\mathcal{W}},x)={\langle x,S_{1\mathcal{W}}x\rangle}={\langle x,\sum_{i=1}^{N}
P_{W_i}(x)\rangle}=\sum_{i=1}^{N}
\|P_{W_i}(x)\|^2=\mathcal{R}_{\mathcal{W}}(x),$$
where $S_{1\mathcal{W}}$ 
is the frame operator with respect to the normalized version of 
$\mathcal{W}= \{(W_i, v_i)\}_{i=1}^{N}.$\\
The Rayleigh quotient is used in the min-max theorem to get exact values of all eigenvalues. It is also used in eigenvalue
algorithms to obtain an eigenvalue approximation. Specifically, this is the basis for Rayleigh quotient iteration.
So, the applications of Rayleigh quotient satisfy for the redundancy function.

With the previously defined notion of lower and  upper redundancies,
  we can verify the main result of this paper.

\begin{prop}\label{3.3}
 Let $\mathcal{W}=\{(W_i, v_i)\}_{i=1}^{N}$ be a fusion frame for $\mathcal{H}^{n}.$
  Then the following statements hold:\\
 \\$[D1]$ The normalized version of $\mathcal{W}$ is an A-tight fusion frame if and only if

 $\mathcal{R}_{\mathcal{W}}^-= \mathcal{R}_{\mathcal{W}}^+=A$

   $ $
\\$[D2]$ $\mathcal{W}$ is an orthonormal fusion basis for $\mathcal{H}^{n}$ if and only if $
\mathcal{R}_{\mathcal{W}}^-= \mathcal{R}_{\mathcal{W}}^+=1 $ and $v_i=1$ for all $i=1,...,N$ .\\
\\$[D3]$ Additivity: For each orthonormal fusion basis $
\mathcal{E}=\{E_i\}_{i=1}^{N}$, we have $$\mathcal{R}_{\mathcal{W}\cup
\mathcal{E}}^{\pm}= \mathcal{R}_{\mathcal{W}}^{\pm}+1.$$ Moreover,
for each fusion frame $\mathcal{V}$ for ${\h}^{n}$  we have
$$\mathcal{R}_{\mathcal{W}\cup \mathcal{V}}^{-}\geq
\mathcal{R}_{\mathcal{W}}^{-}+\mathcal{R}_{\mathcal{V}}^{-} \quad
and\quad\mathcal{R}_{\mathcal{W}\cup \mathcal{V}}^{+}\leq
\mathcal{R}_{\mathcal{W}}^{+}+\mathcal{R}_{\mathcal{V}}^{+}.
$$ In particular, if $ \mathcal{W}$ and $ \mathcal{V}$ have uniform
redundancies, then $$\mathcal{R}_{\mathcal{W}\cup \mathcal{V}}^{-}=
\mathcal{R}_{\mathcal{W}}^++\mathcal{R}_{\mathcal{V}}^+=\mathcal{R}_{\mathcal{W}\cup
\mathcal{V}}^{+} .$$
\\$[D4]$ Invariance: Redundancy is invariant under application of a
unitary operator $U$ on the subspaces $W_i$ of ${\h}^{n},$  i.e., $$\mathcal{R}_
{\mathcal{W}}^{\pm}=\mathcal{R}_{U(\mathcal{W})}^{\pm},$$
\\and under any permutation $\pi\in S_{\{1,...,N\}},$  i.e.,
$$\mathcal{R}_{\mathcal{W}_{\pi_i}}^{\pm}=\mathcal{R}_{\mathcal{W}}^{\pm}
 \quad for \quad all \quad \pi\in S_{\{1,...,N\}}.$$
\end{prop}
\begin{proof}
\textbf[D1] Assume that $\mathcal{W}=\{(W_i, v_i)\}_{i=1}^{N}$ is a fusion frame with uniform redundancy
 $\mathcal{R}_{\mathcal{W}}^-= \mathcal{R}_{\mathcal{W}}^+=A.$
Let $\mathcal{W}'=\{(W_i, 1)\}_{i=1}^{N}$ be the normalized version of $\mathcal{W}$
with bounds $C$ and $D.$ Then, for each $x\in {\h}^{n}$ we have  $$C\|x\|^2\leq\sum_{i=1}^{N}\|P_{W_i}(x)\|^2\;\leq
D\|x\|^2.$$ Now, let  $x\in \SSS,$
so
  $C\leq\mathcal{R}_{\mathcal{W}}(x)\leq
D.$
By the hypothesis, $\mathcal{R}_{\mathcal{W}}^-= \mathcal{R}_{\mathcal{W}}^+=A$ therefore, $C=D=A.$
Consequently, the normalized version of $\mathcal{W}$ is A-tight.\\
The reverse implication is obvious.\\
\\\textbf[D2] This follows from Proposition \ref{111} and condition [D1] above.\\
\\\textbf[D3]
By Lemma \ref{2.9}, the union of each two fusion frames is also a fusion frame.
 Hence $\mathcal{R}_{\mathcal{W}\cup \mathcal{E}}$ is well-defined.

The first claim follows from the definition of
$\mathcal{R}_{\mathcal{W}}$ and Proposition \ref{111}, as follows
$$\mathcal{R}_{\mathcal{W}\cup
\mathcal{E}}(x)=\sum_{i=1}^{N} \|P_{W_i}(x)\|^2+\sum_{i=1}^{N}\|P_{E_i}(x)\|^2=\mathcal{R}_{\mathcal{W}}(x)+1$$ which implies that
$$\mathcal{R}_{\mathcal{W}\cup
\mathcal{E}}^{\pm}= \mathcal{R}_{\mathcal{W}}^{\pm}+1.$$
\\Next, let $\mathcal{W}=\{(W_i, w_i)\}_{i=1}^{N}$ and $\mathcal{V}=\{(V_j, v_j)\}_{j=1}^{M}$ be fusion frames for $\mathcal{H}^{n}$.
 Then for each $x\in \SSS$, we have
$$\mathcal{R}_{\mathcal{W}\cup
\mathcal{V}}(x)=\sum_{i=1}^{N} \|P_{W_i}(x)\|^2+\sum_{j=1}^{M}
\|P_{{{V}}{_j}}(x)\|^2,$$
hence
$$\mathcal{R}_{\mathcal{W}\cup \mathcal{V}}^{-}=\min_{x\in \SSS}\mathcal
{R}_{\mathcal{W}\cup \mathcal{V}}\geq \min_{x\in \SSS}
\mathcal{R}_{\mathcal{W}}(x)+\min_{x\in \SSS}
\mathcal{R}_{\mathcal{V}}(x)=\mathcal{R}_{\mathcal{W}}^-+\mathcal{R}
_{\mathcal{V}}^-,$$ and
$$\mathcal{R}_{\mathcal{W}\cup \mathcal{V}}^{+}=\sup_{x\in \SSS}\mathcal{R}_
{\mathcal{W}\cup \mathcal{V}}\leq \sup_{x\in \SSS}
\mathcal{R}_{\mathcal{W}}(x)+\sup_{x\in \SSS}\mathcal{R}_{\mathcal{V}}(x)
=\mathcal{R}_{\mathcal{W}}^++\mathcal{R}_
{\mathcal{V}}^+.$$
For the particular case, we have
$$\mathcal{R}_{\mathcal{W}\cup \mathcal{V}}^{-}
\geq\mathcal{R}_{\mathcal{W}}^-+\mathcal{R}_{\mathcal{V}}^-
=\mathcal{R}_{\mathcal{W}}^++\mathcal{R}_{\mathcal{V}}^+\geq
\mathcal{R}_{\mathcal{W}\cup \mathcal{V}}^{+},$$
and since in general
$$\mathcal{R}_{\mathcal{W}\cup \mathcal{V}}^{-}\leq
\mathcal{R}_{\mathcal{W}\cup \mathcal{V}}^{+},$$ therefore
$$\mathcal{R}_{\mathcal{W}\cup \mathcal{V}}^{-}=\mathcal{R}_{\mathcal{W}\cup
\mathcal{V}}^{+}.$$
\\\textbf[D4] Let $U$ be a unitary operator on $\mathcal{H}^{n}$.
 Proposition \ref{aaa} implies that  $U\mathcal{W}=\{(UW_i,v_i)\}_{i=1}^{N}$ is a fusion
 frame for $\mathcal{H}^{n}.$ Let $x\in\SSS.$
 The redundancy function for $U\mathcal{W}$ is as follows

\begin{eqnarray*}
\mathcal{R}_{U\mathcal{W}}(x)&=&\sum_{i=1}^{N}
\|P_{UW_i}(x)\|^2\\
 &=&\sum_{i=1}^{N}
\|UP_{W_i}{U}^{-1}(x)\|^2\\
& =&\sum_{i=1}^{N}
\|P_{W_i}({U}^{*}x)\|^2\\
&=&\sum_{i=1}^{N}
\|P_{W_i}(x')\|^2\\
&=&\mathcal{R}_\mathcal{W}(x'),
\end{eqnarray*}
where the above equalities follow from the fact that $P_{UW_i}=UP_{W_i}{U}^{-1}$ and $U$ is unitary.
  Hence, redundancy is invariant under application of a
unitary operator $U$ on the subspaces $\{W_i\}_{i=1}^{N}$  of ${\h}^{n}$.
 Invariance under the permutations of the subspaces $\{W_i\}_{i=1}^{N}$ is clear.
\end{proof}

In the case of ordinary frames, the redundancy of a Riesz basis is
exactly equal to one \cite{1}. In the following corollary we show that  this
fact holds also for a Riesz decomposition of
${\h}^{n}.$
\begin{cor}
Let $\mathcal{W}=\{(W_i,v_i)\}_{i=1}^{N}$ be a Riesz decomposition of
${\h}^{n}.$
Then the redundancy of $\mathcal{W}$ is equal to 1.
\end{cor}
\begin{proof}
Since $\mathcal{W}=\{(W_i,v_i)\}_{i=1}^{N}$ is a Riesz decomposition of
${\h}^{n}$, then it is minimal, by Proposition \ref{bbb}.
Let $x\in \SSS$ be arbitrary. Because of  the minimality of $\mathcal{W},$ the element $x$ can not be in
two subspaces of $\{W_i\}_{i=1}^{N},$ simultaneously.
Hence, there exists a unique $i_0\in \{1,...,N\}$ such that $x\in W_{i_0},$ from which
it follows that
$$\mathcal{R}_\mathcal{W}(x)=\|P_{W_{i_0}}(x)\|^2=\|x\|^2=1.$$
This claim satisfies for all $x\in \SSS,$ so $\mathcal{R}_\mathcal{W}=1.$

\end{proof}

A crucial question concerns the change of redundancy once an invertible
operator is applied to a fusion frame, which we state it as follows.\\
\begin{cor}\label{3.5}
Let $\mathcal{W}=\{(W_i,v_i)\}_{i=1}^{N}$ be a fusion frame for ${\h}^{n}$.
 For every invertible operator $T$ on ${\h}^{n},$ we have
$$\mathcal{R}_{\mathcal{W}}^{\pm}(k(T))^{-2}\leq \mathcal{R}_{T(\mathcal{W})}^{\pm}
\leq \mathcal{R}_{\mathcal{W}}^{\pm}(k(T))^{2},$$
where $k(T)=\|T\|\|T^{-1}\|$ denotes the condition number of $T$.
\end{cor}
\begin{proof}
The proof follows by  Proposition \ref{aaa} and applying the definition of the redundancy function for the
  fusion frame $\{(TW_i,v_i)\}_{i=1}^{N}$.
\end{proof}
A fusion frame is not uniquely specified by its redundancy function.  Since
 we can apply a unitary operator $U$ to $\{W_i\}_{i=1}^{N}$, yet the
  redundancy  being invariant.
 Let $U$  be a unitary operator on ${\h}^{n}$. By Proposition \ref{aaa},
${U\mathcal{W}}=\{(UW_i,v_i)\}_{i=1}^{N}$ is a fusion frame and the fusion
frame operator $S_{U\mathcal{W}}=US_{\mathcal{W}}U^{-1} $. We denote the
fusion frame operator $S_{U\mathcal{W}}$ by
$\tilde{S}_{\mathcal{W}}$ and denote the fusion frame operator with respect to the
 normalized version of ${U\mathcal{W}}=\{(UW_i,v_i)\}_{i=1}^{N}$
by $\tilde{S_{1}}_{\mathcal{W}}.$
  So for any $x\in \SSS$
  $$\tilde{S_{1}}_{\mathcal{W}}(x)=\sum_{i=1}^{N} P_{UW_i}(x) =
  \sum_{i=1}^{N} UP_{W_i}U^{-1}(x)$$
  therefore
\begin{eqnarray*}
 \langle\tilde{S_{1}}_{\mathcal{W}}(x),x\rangle&=&\langle\sum_{i=1}^{N}
  UP_{W_i}U^{-1}(x),x\rangle=\sum_{i=1}^{N} \|P_{W_i}(U^{*}x)\|^2\\
  &=&\sum_{i=1}^{N} \|P_{W_i}({x}')\|^2=\mathcal{R}_{\mathcal{W}}({x}'),
  \end{eqnarray*}
  where $ U^{*}(x)={x}'\in \SSS$ because   $U$
is unitary.
 Now, we can say that, if $\mathcal{W}$ and ${\mathcal{V}}$ are two fusion frames for ${\h}^{n}$
 with associated frame operators $\tilde{S}_{\mathcal{W}}$ and $\tilde{S}_{\mathcal{V}}$ with
 respect to a unitary operator $U,$ then
 $$\tilde{S_{1}}_{\mathcal{W}}=\tilde{S_{1}}_{\mathcal{V}} \quad  on\quad  {\h}^{n} \quad \Rightarrow \quad  \mathcal{R}_
 {\mathcal{W}}=\mathcal{R}_{\mathcal{V}} \quad  on \quad \SSS. $$
 By the definition of the redundancy function, we have a general statement:
 $$\mathcal{R}_{\mathcal{W}}=\mathcal{R}_{\mathcal{V}}
 \quad  on \quad \SSS \quad \Leftrightarrow \quad {S_{1}}_{\mathcal{W}}={S_{1}}_{\mathcal{V}} \quad on \quad {\h}^{n}, $$
in which ${S_{1}}_{\mathcal{W}}$ and ${S_{1}}_{{\mathcal{V}}}$ are the
 fusion frame operators with respect to the normalized version of the fusion frames $\mathcal{W}$
and ${\mathcal{V}}$, respectively.\\
This argument  leads to  define an equivalence relation:
\begin{defn}
The families of all closed subspaces $\{W_i\}_{i=1}^{N}$ of ${\h}^{n}$ which
 construct a fusion frame with respect to weights $\{v_i\}_{i=1}^{N}$
 is called the $admissible$ subspaces and is denoted  by $\mathcal{FF}.$
 \end{defn}
Putting the above results together, we obtain the next proposition immediately.
\begin{prop}
Let $\mathcal{FF}$ be the family of admissible subspaces
with respect to the weights $\{v_i\}_{i=1}^{N}$. Then the relation $\sim$ on
 $\mathcal{FF}$ defined by
$$\mathcal{W}\sim{\mathcal{V}}\Leftrightarrow \mathcal{R}_{\mathcal{W}}=\mathcal{R}_{\mathcal{V}}$$
is an equivalence relation.
\end{prop}
Let $U$ be a unitary operator  on ${\h}^{n}$ and $\mathcal{W}\in \mathcal{FF}.$
So, by condition [D4] of Proposition \ref{3.3}, we have $U\mathcal{W}\sim \mathcal{W}.$\\

 In \cite{Dykema}, Dykema at al. stated some projection decompositions for positive operators. Also they
 proved that every positive invertible operator is the frame operator for a spherical frame.
 By getting an idea from these facts, we will characterize the fusion frame
 operator of an equi-dimensional fusion frame, i.e., a fusion frame with $\dim W_i=m$ for $i=1,...,N.$
\begin{prop}
Let $T$ be a positive invertible operator on ${\h}^{n}$
with discrete spectrum having $mk$ $(m,k\in \NN)$ strictly positive eigenvalues, each repeated a multiple of $m$ times.
Then there exists an equi-dimensional  fusion frame $\mathcal{W}=\{(W_i,v_i)\}_{i=1}^{N}$
 such that $T=S_{1\mathcal{W}},$ where $S_{1\mathcal{W}}$  is the fusion frame operator with respect
 to the normalized version of $\mathcal{W}.$
\begin{proof}
Let $T$ be a positive invertible operator with desired
 properties in hypothesis.
By Lemma 7 in \cite{Dykema}, the operator $T$ can be written as the sum of $N$
rank-m projections provided that
$\operatorname{tr} [T]=mN.$ Hence, there exists an m-dimensional fusion
frame $\mathcal{W}=\{(W_i,v_i)\}_{i=1}^{N}$ for ${\h}^{n}$
 having $T$ as its fusion frame operator with respect
  to its normalized version.
Therefore $T=S_{1\mathcal{W}}$ as claimed.
\end{proof}
\end{prop}
\subsection{Redundancy of Infinite Fusion Frames}
Our approach in redundancy of fusion frames does not capture more information about
infinite fusion frames whose their normalized version is not convergent.
For redundancy of infinite frames, this problem appears for unbounded frames. This problem
with the notion of redundancy of infinite frames was studied in \cite{cahill} by Cahill  at
al. They assumed that $\mathcal{R}_{\Phi}^+< \infty,$
 where $\Phi=\{\varphi_i\}_{i=1}^{\infty}$ is a frame for a Hilbert space $\h.$
 Similarly, we put a restriction  on $\mathcal{R}_{\mathcal{W}}^+$. First we present the definition of redundancy function for an infinite fusion frame.
 \begin{defn}
  Let $I$ be a infinite countable index set and $\h$ be a real or complex Hilbert space.
Assume that $\mathcal{W}=\{(W_i,v_i)\}_{i\in I}$
is a fusion frame for $\h.$
The redundancy function is defined on the unit sphere $\SSS=\{x\in \h: \|x\|=1\}.$ Likewise the finite case
in Section 3.1,
 for each $x\in \SSS$, the redundancy function
$\mathcal{R}_{\mathcal{W}}:\SSS\to \RR^+$ is defined by $$
\mathcal{R}_{\mathcal{W}}(x)=\sum_{i\in I}
\|P_{W_i}(x)\|^2
, \quad x\in \SSS.$$
The lower and upper redundancies defined similar to finite case.
 \end{defn}

The redundancy function may not assume its minimum or maximum on the unit sphere and in general
both the min and max of this function could be infinite.
It  is sufficient that we assume that $\mathcal{R}_{\mathcal{W}}^+<\infty.$ Then all of the
results in Section 3.1 but Proposition 3.8 are satisfy for the fusion frame
$\mathcal{W}=\{(W_i,v_i)\}_{i\in I}.$

\section{ Redundancy of Dual Frames}
Dual frames play an important role in  studying frames and their applications,
 specially in the reconstruction formula. Therefore it is natural to study and consider
  their redundancy and its relationship with redundancy of the original frame. In this
  section, we will show that the ratio between redundancies of frames (fusion frame) and dual frames (dual fusion frames) is
  bounded from below and above by some significant numbers.  First, we will review
  the definition of the redundancy function for finite frames \cite{2}, and the definition of
  dual frames.
\begin{defn}\cite{2}
Let $\Phi=\{\varphi_i\}_{i=1}^{N}$ be a frame for a finite dimensional Hilbert space ${\h}^{n}$.
For each $x\in \SSS$, the redundancy function
$$\mathcal{R}_\Phi:\SSS\to \RR^+$$ is defined by

$$\mathcal{R}_\Phi(x)=\sum_{i=1}^{N}
\|P_{\langle\varphi_i\rangle}(x)\|^2,$$
where $\langle\varphi_i\rangle$ denotes the span of $\varphi_i \in \h$ and
 $P_{\langle\varphi_i\rangle}$ denotes the orthogonal projection onto $\langle\varphi_i\rangle.$
 The \emph{upper redundancy} of $\Phi$ is defined by
$$\mathcal{R}_{\Phi}^+=\sup_{x\in\SSS}
\mathcal{R}_{\Phi}(x), $$ and the \emph{lower redundancy} of
$\Phi$  by $$\mathcal{R}_{\Phi}^-=\inf_{x\in \SSS}
\mathcal{R}_{\Phi}(x). $$ The frame $\Phi$ has $uniform$
redundancy if $ \mathcal{R}_{\Phi}^-=
\mathcal{R}_{\Phi}^+$.
\end{defn}
\begin{defn}
Given a frame $\Phi=\{\varphi_i\}_{i=1}^{N}$ for ${\h}^{n}$, another frame
$\Psi=\{\psi_i\}_{i=1}^{N}$ is said to be a dual frame of $\Phi$ if the following holds:
$$x=\sum_{i=1}^{N}\langle x,\varphi_i\rangle\psi_i \quad for \quad all\quad x\in {\h}^{n}.$$
If we denote by $S_{\Phi}$ the frame operator of $\Phi,$ then the frame $\{{S_{\Phi}}^{-1}\varphi_i\}_{i=1}^{N}$
is called the canonical dual of $\Phi.$
By the classical results in \cite{10}, all duals to a frame $\Phi$
can be expressed as
$$\big \{{S_{\Phi}}^{-1}\varphi_i+\eta_i-\sum_{k=1}^{N}\langle {S_{\Phi}}^{-1}\varphi_i,\varphi_k\rangle\eta_k \big\}_{i=1}^{N}$$
where $\eta_i\in {\h}^{n}$ for $i=1,...,N$ is arbitrary.
Dual frames which do not coincide with the canonical dual frame are often coined alternate dual frame.
\end{defn}
Similar to the Corollary \ref{3.5}, we have the following result for the redundancies of frames and their canonical duals.
\begin{prop}
Let $\Phi=\{\varphi_i\}_{i=1}^{N}$ be a frame for ${\h}^{n}$
 and $\Psi={S_\Phi}^{-1}\Phi$
be the canonical dual of $\Phi$. Then
$$\mathcal{R}_{\Phi}^{\pm}(k(S_{\Phi}))^{-2}\leq \mathcal{R}_{\Psi}^{\pm}
\leq \mathcal{R}_{\Phi}^{\pm}(k(S_{\Phi}))^{2},$$
where $k(S_{\Phi})=\|S_{\Phi}\|\|S_{\Phi}^{-1}\|$
denotes the condition number of the frame operator $S_{\Phi}.$

\end{prop}
\begin{proof}
Note that the redundancy function for finite frames
 is equal to the redundancy function for a finite
fusion frame with  $\dim W_i=1$ for all $i=1,...,N.$ For the proof, it is enough to apply the Corollary \ref{3.5}.

\end{proof}
\indent We generalize a result for tight frames.
First, we state a lemma from \cite{10}.
\begin{lem}\label{4.3}\cite{10}
Let $\Phi=\{\varphi_i\}_{i=1}^{N}$ be a frame for ${\h}^{n}$.
Then the following are equivalent.
 \begin{enumerate}
\item $\Phi=\{\varphi_i\}_{i=1}^{N}$ is tight;
\item $\Phi=\{\varphi_i\}_{i=1}^{N}$ has a dual of the
form $\{\psi_i\}_{i=1}^{N}=\{C\varphi_i\}_{i=1}^{N}$
 for some constant $C>0.$
\end{enumerate}

 \end{lem}
 It is well known that the redundancy is invariant under scaling. So,  using the  above lemma, we have:

\begin{prop}
Let $\Phi$ be a tight frame for ${\h}^{n}$,
and $\Psi={S_\Phi}^{-1}\Phi$ be
 its canonical dual frame. Then, for each $x\in \SSS,$
 $$\mathcal{R}_\Phi(x)=\mathcal{R}_\Psi(x).$$
\end{prop}
For general frames (not necessary tight frames) and their
 alternate duals, we have a relationship between their redundancies
 in a particular case.
 First, we  state a lemma from \cite{5}.
 \begin{lem}\cite{5}
Let $\Phi=\{\varphi_i\}_{i=1}^{N}$ be a frame for ${\h}^{n}$
with the canonical dual
 ${S_\Phi}^{-1}\Phi.$ If $\Psi=\{\psi_i\}_{i=1}^{N}$
 is the alternate dual of $\Phi,$
 then
 $${\|(\langle x,{S_\Phi}^{-1}{\varphi_i}\rangle)_{i=1}^{N}}\|_2\leq
 {\|(\langle x,\psi_i\rangle)_{i=1}^{N}}\|_2.$$
 \end{lem}
 \indent In particular, suppose that ${S_\Phi}^{-1}\Phi$ and $\Psi=\{\psi_i\}_{i=1}^{N}$
 in the  previous lemma be equal norm
 frames, i.e.,  $\|{S_\Phi}^{-1}{\varphi_i}\|=c$ and $\|\psi_i\|=d$ for some $c,d>0$ for
 $i=1,...,N$. Then, for each $x\in \SSS,$

\[
\mathcal{R}_{{S_\Phi}^{-1}\Phi}(x)=\sum_{i=1}^{N}{\|{S_\Phi}^{-1}{\varphi_i}\|}^{-2}{|\langle x,{S_\Phi}^{-1}{\varphi_i}\rangle|}^{2}
=
 {c}^{-2}\sum_{i=1}^{N}{|\langle x,{S_\Phi}^{-1}{\varphi_i}\rangle|}^{2}
\]
and
\[
\mathcal{R}_\Psi(x)=\sum_{i=1}^{N}{\|\psi_i\|}^{-2}{|\langle x,\psi_i\rangle|}^{2}
=
 {d}^{-2}\sum_{i=1}^{N}{|\langle x,\psi_i\rangle|}^{2}
 \]
 so
 $$\mathcal{R}_{{S_\Phi}^{-1}}(x)\leq ({d\over c})^{2}\mathcal{R}_\Psi(x).$$

Now we state similar results for fusion frames. For the fusion frame $\mathcal{W}=\{(W_i,v_i)\}_{i\in I}$
we will assume that $\mathcal{R}_{\mathcal{W}}^+<\infty.$
 First we review the definition of the
$canonical$ $dual$ fusion frames.
\begin{defn}\cite{6}
If $\mathcal{W}=\{(W_i,v_i)\}_{i\in I}$ is a fusion frame for $\h$,
then ${S_\mathcal{W}}^{-1}\mathcal{W}=\{({S_\mathcal{W}}^{-1}W_i,v_i)\}_{i\in I}$
is called the $canonical$ $dual$ fusion frame.
\end{defn}
By Proposition \ref{aaa}, if $\mathcal{W}=\{(W_i,1)\}_{i\in I}$ is
 a fusion frame for $\h$ with bounds $A$ and $B,$ then its
 canonical dual is a fusion frame with bounds
  $${A\over {\|S_\mathcal{W}\|^{2}\|{S_\mathcal{W}}^{-1}\|^{2}}} \quad and \quad
B\|S_\mathcal{W}\|^{2}\|{S_\mathcal{W}}^{-1}\|^{2}.$$\\
Therefore the ratio between the redundancies  of the fusion frame $\mathcal{W}$
and its canonical dual is as follows:
$${{{A}^{3}}\over {B}}\leq {{\mathcal{R}_\mathcal{W}}\over {\mathcal{R}_{{S_\mathcal{W}}^{-1}
\mathcal{W}}}} \leq {{{B}^{3}}\over {A}}.$$
In particular, if $\mathcal{W}=\{(W_i,1)\}_{i\in I}$ is a tight fusion frame,
then $\mathcal{W}$ is equal to its canonical dual  ${S_\mathcal{W}}^{-1}\mathcal{W}$.
Hence for $A$-tight fusion frames, we have
$$\mathcal{R}_\mathcal{W}(x)=\mathcal{R}_{{S_\mathcal{W}}^{-1}\mathcal{W}}(x)=A.$$
The alternate dual of fusion frames were introduced in \cite{11}.
\begin{defn}\cite{11}
Let $\mathcal{W}=\{(W_i,w_i)\}_{i\in I}$ be a fusion frame for $\h$ with the fusion frame
operator $S_\mathcal{W}$. The fusion Bessel sequence ${\mathcal{V}}=\{({V}_i,v_i)\}_{i\in I}$
 is called an $alternate$
$dual$ fusion frame  for $\mathcal{W}$ where
$$x=\sum_{i\in I}v_iw_iP_{{V}_i}{S}^{-1}_\mathcal{W}P_{{W}_i}(x),\quad for\quad all\quad x\in \h.$$
\end{defn}
\begin{prop}\cite{11}
Let $\mathcal{W}=\{(W_i,w_i)\}_{i\in I}$ be a fusion frame for $\h$ with bounds $A,B$
 and ${\mathcal{V}}=\{({V}_i,v_i)\}_{i\in I}$ be the alternate dual of $\mathcal{W}$ with Bessel
  bound $C$. Then ${\mathcal{V}}$ is also a fusion frame with bounds ${1}\over {{B{\|{S}^{-1}_\mathcal{W}\|}^{2}}}$ and $C.$
\end{prop}

Similar to the ordinary frames, it is easy to check  the following
 relation between redundancies of a fusion frame and its alternate dual.\par
 Let $\mathcal{W}=\{(W_i,1)\}_{i\in I}$ be  a fusion frame for $\h$
  with bounds $A,B$ and with the alternate
  dual ${\mathcal{V}}=\{({V}_i,1)\}_{i\in I}.$ Then
  $${{{1}}\over {\|{S}^{-1}_\mathcal{W}\|}^{2}}\leq {{\mathcal{R}_{\mathcal{V}}}\over
   {\mathcal{R}_{\mathcal{W}}}} \leq {{C}\over {A}},$$
  where $S_\mathcal{W}$ is the frame
operator for $\mathcal{W}$ and  $C$ is the upper fusion frame bound for $\mathcal{V}.$\\
\section{Excess of Fusion Frames}
 In an analogous  way to the frame theory, the concept of
 excess was introduced for fusion frames in \cite{12}.
We restate some results in the excesses of fusion frames and reprove them in an easier way. We draw
 similar  comparison about  redundancy of fusion frames.
\begin{defn}\cite{12}
Let $\mathcal{W}=\{(W_i,v_i)\}_{i\in I}$  be a fusion frame for $\h$ with synthesis operator ${T_\mathcal{W}}$.
The $excess$ of $\mathcal{W}$ is defined as $$e(\mathcal{W})=\dim N(T_\mathcal{W}),$$ where $N(T_\mathcal{W})=\ker [T_\mathcal{W}]$.
\end{defn}
\begin{defn}\cite{6}
Let $\{W_i\}_{i\in I}$ and $\{V_i\}_{i\in I}$ be fusion frames with
respect to the same family of weights.
We say that they are $unitary$ $equivalent$ if there exists a unitary
 operator $U$ on $\h$ such that $W_i=U(V_i)$.
\end{defn}
We have a similar statement for the family of admissible weights \cite{12}.

\begin{defn}\cite{12}
Let $\mathcal{W}=\{W_i\}_{i\in I}$ be a generating sequence  of closed subspaces of $\h,$  i.e.,
 ${\overline {\operatorname{span}} }\{W_i, i\in I\}=\h.$
 Let $\mathcal{P}(\mathcal{W})$ be  the set of weights $\{w_i\}_{i\in I}\in \ell^{\infty}_+(I)$
 such that $\mathcal{W}=\{(W_i,w_i)\}_{i\in I} $ be a fusion frame for $\h.$\\

$\mathcal{P}(\mathcal{W})$ is called $admissible$  weights for $\mathcal{W}$.
 Given $v,w\in \mathcal{P}(\mathcal{W})$, we say that $v$ and $w$ are
  $equivalent$ if there exists $\alpha>0$ such that
 $v=\alpha w$.
\end{defn}
\begin{prop}
Let $\mathcal{W}=\{(W_i,w_i)\}_{i\in I}$ and $\mathcal{V}=\{(V_i,v_i)\}_{i\in I}$
 be  unitary equivalent
fusion frames. Then $e(\mathcal{W})=e(\mathcal{V})$.
\end{prop}
\begin{proof}
Assume that $\mathcal{V}=U\mathcal{W}$ for some unitary operator $U$ on $\h$ and let
$T_\mathcal{W}$ and $T_{\mathcal{V}}$ be the synthesis operators of $\mathcal{W}$ and $\mathcal{V}$, respectively. We should
prove that
$$\dim N(T_\mathcal{W})= \dim N(T_{\mathcal{V}}).$$
Since ${V}_i=UW_i=\big \{ \{g_i\}_{i\in I}: g_i=Uf_i, f_i\in W_i\big \}$,  the synthesis operator $\mathcal{V}$ will be
$$T_{\mathcal{V}}:(\sum_{i\in I}{\oplus UW_i})_{\ell_2} \to \h$$
  $$  T_{\mathcal{V}}(g)= \sum_{i\in I}{v_i g_i},\quad  g=\{g_i\}_{i\in I}\in (\sum_{i\in I}{\oplus UW_i})_{\ell_2}.$$
   Now we have
  \begin{eqnarray*}
 N(T_{\mathcal{V}})&=& \big \{\{g_i\}_{i\in I}: T_{\mathcal{V}}
  (\{g_i\})=0 \big \}\\
  &=&\big \{\{g_i\}_{i\in I}: \sum_{i\in I} v_ig_i=0\big \}\\
 &=&\big \{\{Uf_i\}_{i\in I}: \sum_{i\in I} v_iUf_i=0\big \}\\
 &=&U\big \{\{f_i\}_{i\in I}:  U\sum_{i\in I} v_if_i=0\big \}\\
  &=&U\big \{\{f_i\}_{i\in I}:  \sum_{i\in I} v_if_i=0\big \}\\
  &=&U(N(T_\mathcal{W})).
  \end{eqnarray*}
  Since
 $U$ is unitary,  we have
  $$\dim N(T_{\mathcal{V}})=\dim U(N(T_{\mathcal{W}}))=\dim N(T_{\mathcal{W}}),$$
  therefore
   $$e(\mathcal{W})=e(\mathcal{V}).$$
\end{proof}
For equivalent admissible weights, we have the following proposition.
\begin{prop}
Let $\mathcal{W}=\{(W_i,v_i)\}_{i\in I}$ be a fusion frame
 and $\{{w_i\}_{i\in I}}\in {\mathcal{P}(\mathcal{W})}$ be equivalent to $\{v_i\}_{i\in I}$.
Then $\mathcal{V}=\{(W_i,w_i)\}_{i\in I}$ is a fusion frame and $e(\mathcal{W})=e(\mathcal{V}).$
\end{prop}
\begin{proof}
Similar to the proof of the above proposition,  it is straightforward that
$N(T_\mathcal{V})=N(T_\mathcal{W}),$ so $e(\mathcal{W})=e(\mathcal{V}).$
\end{proof}
As in the case of ordinary frames, we have the following result for Riesz decomposition of $\h$.
\begin{cor}
Let $\mathcal{W}=\{(W_i,v_i)\}_{i\in I}$ be a Riesz decomposition of
$\h$.
Then the excess of $\mathcal{W}$ is equal to zero.
\end{cor}
\begin{proof}
Since $\mathcal{W}=\{(W_i,v_i)\}_{i\in I}$ is a Riesz decomposition of
$\h$, therefore by Proposition \ref{bbb}, the synthesis operator $T_{\mathcal{W}}$ is one to one. Hence $$e(\mathcal{W})=\dim N({T_{\mathcal{W}}})=0.$$
\end{proof}

Similar comparisons hold between redundancies of fusion frames $\mathcal{W}$ and
$U\mathcal{W}$ for a unitary operator $U$ on $\h,$ see Proposition \ref{3.3} part [D4].\\
For equivalent weights we have the following proposition that its proof is obvious.
\begin{prop}
Let $\mathcal{W}=\{(W_i,v_i)\}_{i\in I}$ be a fusion frame for $\h$ and
let
$\mathcal{V}=\{(W_i,\alpha v_i)\}_{i\in I}$, for some $\alpha>0.$
Then $\mathcal{V}$ is a fusion frame and the redundancies of $\mathcal{W}$ and $\mathcal{V}$
 satisfy
$$\mathcal{R}_\mathcal{W}=\mathcal{R}_\mathcal{V}.$$
\end{prop}

\section{Redundancy of Fusion Frame Systems}
For a given fusion frame, one can build local frames for each
 subspace and putting them together to get global frames. Therefore the connection between the redundancies of local frames
 and the original fusion frame shall be interesting. In this section,
  we will  study this connection in some special cases.
\begin{defn}\cite{8}
Let $\mathcal{W}=\{(W_i,v_i)\}_{i\in I}$ be a fusion frame for $\h$
and let $\{f_{ij}\}_{j\in J_i}$
 be a frame for $W_i$  for each $i\in I$. Then we call $\{(W_i,v_i,
 \{f_{ij}\}_{j\in J_i})\}_{i\in I}$
 a $fusion$ $frame$ $system$ for $\h$. The frame vectors $\{f_{ij}\}_{j\in J_i}$
  are called local frame vectors.
 \end{defn}
 The next proposition shows the connection between the redundancies of the local frames
 and  fusion frame in the fusion frame system  $\mathcal{W}=\{(W_i,v_i,
 \{f_{ij}\}_{j\in J_i})\}_{i\in I},$ in a special case.
 \begin{prop}

Let $\mathcal{W}=\{(W_i,v_i)\}_{i=1}^{N}$ be a fusion frame for $\h^{n}$
and let $\phi_i=\{f_{ij}\}_{j=1}^{m_i}$
 be a finite frame for $W_i$, for $i=1,...,N$.   Assume that each frame  $\phi_i$ has orthogonal elements, for $i=1,...,N.$

If $\mathcal{R}_{\phi_i}(x)$ and  $\mathcal{R}_\mathcal{W}(x)$
denote the redundancy functions of frames
  $\phi_i$ for $i=1,...,N$ and the fusion frame $\mathcal{W} $ respectively,
  then
  $$\mathcal{R}_\mathcal{W}(x)=\sum_{i=1}^{N} \mathcal{R}_{\phi_i(x)}\quad for \quad all \quad x\in \SSS.$$
 \end{prop}
 \begin{proof}
Let $\phi_i=\{f_{ij}\}_{j=1}^{m_i}$ be a finite frame for $W_i,$ so $W_i={\operatorname{span}}\{f_{i1},...,f_{im_{i}}\}$
 for $i=1,...,N$.
 Therefore the orthogonal projections $P_{W_i}$ are
 $$P_{W_i}(x)=c_1f_{i1}+...+c_{m_{i}}f_{im_{i}},\quad i=1,...,N$$
 where, $c_j=$ ${\langle x,f_{ij}\rangle}\over
 {{\|f_{ij}\|}^{2}}$ $, j=1,...,m_i.$
 Without loss of generality,  suppose that $\|f_{ij}
 \|\neq 0$ for $i=1,...,N$, $j=1,...,m_i.$
 By the definition of the redundancy function for finite
 frames, we have
 $$\mathcal{R}_{\phi_i}(x)=\sum_{j=1}^{m_i}{\|P_{\langle f_{ij}\rangle}(x)\|}^{2}=\sum_{j=1}^{m_i}{\|
 {{\langle x,f_{ij}\rangle}\over {{\|f_{ij}\|}^{2}}}} {f_{ij}}\|^{2}.$$
 The redundancy function for the fusion frame
 $\mathcal{W}=\{(W_i,v_i)\}_{i=1}^{N}$ is
 $$\mathcal{R}_\mathcal{W}(x)=\sum_{i=1}^{N}\|P_{W_i}(x)\|^2=
 \sum_{i=1}^{N}\|c_1f_{i1}+...+c_{m_{i}}f_{im_{i}}\|^2
$$$$=\sum_{i=1}^{N}\sum_{j=1}^{m_i}{\|
 {{\langle x,f_{ij}\rangle}\over {{\|f_{ij}\|}^{2}}}}
  {f_{ij}}\|^{2}
=\sum_{i=1}^{N}\mathcal{R}_{\phi_i}(x).$$
 \end{proof}

Parseval frames play an important role in abstract frame theory,
since they are extremely useful for applications.
For Parseval fusion frames we have the following characterization \cite{6}.\\
\begin{lem}\cite{6}
For each $i\in I,$ let $v_i>0$ and let $\{f_{ij}\}_{j\in{J_i}}$ be a
Parseval frame sequence in $\h.$
Define $W_i=\overline{\operatorname{span}}\{f_{ij}\}_{j\in{J_i}}$ for all $i\in I$
and choose for each subspace $W_i$
an orthonormal basis $\{e_{ij}\}_{j\in{J_i}}.$
Then the following conditions are equivalent.
\begin{enumerate}
\item $\{v_if_{ij}\}_{i\in I,j\in{J_i}}$ is a Parseval frame for $\h$;
\item $\{v_ie_{ij}\}_{i\in I,j\in{J_i}}$ is a Parseval frame for $\h$;
\item $\{(W_i,v_i)\}_{i\in I}$ is a Parseval fusion frame for $\h.$
\end{enumerate}

\end{lem}
Putting the above lemma, Proposition \ref{111} and Proposition \ref{3.3} together,
we get the following corollary immediately.
\begin{cor}

Let $\{f_{ij}\}_{j\in{J_i}}$ be a
Parseval frame sequence in $\h.$
Define $W_i=\overline{\operatorname{span}}\{f_{ij}\}_{j\in{J_i}}$ for all $i\in I.$
Then the following conditions are equivalent:
\begin{enumerate}
\item $\{f_{ij}\}_{i\in I,j\in{J_i}}$ is a Parseval frame for $\h.$\\
\item $\{(W_i,1)\}_{i\in I}$ is a fusion frame for $\h$ with redundancy equal to $1$.
\end{enumerate}
\end{cor}

\section{Examples}
Now we analyze the fusion frames $\mathcal{W}$ and  $\mathcal{V}$ introduced in Section 1.
We will show that
the lower and  upper redundancies are precisely equal to those values which we expected.

\begin{exam}
 The family  $\mathcal{W}=\{W_1,...,W_1,W_2,...,W_n\}$ where $W_1$ occurs $n+1$ times
 is a fusion frame with respect to $v_i=1$ for all $i=1,...,n.$ The fusion frame
 $\mathcal{W}$ satisfies
$$\mathcal{R}_{\mathcal{W}}^-=1 \quad and \quad \mathcal{R}_{\mathcal{W}}^+=n+1.$$
This can be seen as follows. Let $x\in \SSS.$ By the definition of the redundancy function for fusion frames and condition [D3] from Proposition 3.3,
$$\mathcal{R}_{\mathcal{W}}(x)=n\|P_{W_1}(x)\|^2+\sum_{i=1}^{n} \|P_{W_i}(x)\|^2$$
$$=n{\|\langle x,e_1\rangle\|}^{2}+1 \leq 1+n.$$
Now, let $x=e_2$. We have
$$\mathcal{R}_{\mathcal{W}}(e_2)=1\leq \mathcal{R}_{\mathcal{W}}(x), \quad
 \forall x\neq e_2$$
which implies that $\mathcal{R}_{\mathcal{W}}^-=1$. Exploiting [D1]
and [D2] from Proposition 3.3, the fusion frame $\mathcal{W}$ is neither orthonormal fusion basis
nor tight.

The fusion frame $\mathcal{V}=\{W_1,W_1,W_2,W_2,...,W_n,W_n\}$
from Section 1, possesses a uniform redundancy.
More precisely,
$$\mathcal{R}_{\mathcal{V}}^-=\mathcal{R}_{\mathcal{V}}^+=2.$$
This follows from
$$\mathcal{R}_{\mathcal{V}}(x)=2\sum_{i=1}^{n} \|P_{W_i}(x)\|^2=2,$$
in which the last equality follows from the fact that $\{(W_i,v_i)\}_{i=1}^{n}$ is
 an 1-Parseval fusion frame. Hence, $\mathcal{R}_{\mathcal{V}}$ takes its
  minimum and   maximum over the unit sphere $\SSS$.
\\Note that $\mathcal{V}$ is a $2$-tight fusion frame therefore by part [D1] from Proposition 3.3,
 the uniform redundancy coincides with the customary notion of redundancy
  as  the following quotient:
 $$\sum_{i=1}^{2n}{ v^2_i \dim W_i \over {n}
}={2n\over n}=2.$$
\end{exam}
\begin{exam}
Let $\mathcal{W}=\{W_1,...,W_n\},$ where $W_i$ comes from example 7.1, for $i=1,...,n.$
Then it is clear that $\mathcal{W}$ is an orthonormal fusion basis for $\h.$
 Hence $\mathcal{R}_{\mathcal{W}}=1.$ It is obvious that  $\mathcal{W}$
 is not robust against any erasures.
 \end{exam}
In the next  example, $\mathcal{R}_{\mathcal{W}}=1$ but $\mathcal{W}$ is robust against $1$-erasure of each subspace.
So, the robustness of a fusion frame, depends on choosing the subspaces and weights.
\begin{exam}
Let $\{e_i\}_{i=1}^{5}$ be an orthonormal basis for
 $\CC^{5}.$ Consider  $W_1:=\operatorname{span}\{e_1, e_2,e_3\},$
 $W_2 :=\operatorname{span}\{e_2,e_3,e_4\},$ $W_3:=\operatorname{span}\{e_4,e_5\}$ and let  $W_4:=\operatorname{span}\{e_1,e_5\}$  with weights
$v_1=v_3={\sqrt{2\over 3}},$ $v_2$ $=v_4$ $=$ ${2\sqrt{3}}\over 3$. Then $\mathcal {W}=\{W_i,v_i\}_{i=1}^{4}$
is a $2$-tight fusion frame for $\CC^{5}.$
By Proposition \ref{8}, two subspaces from $\mathcal {W}$ can be deleted yet leaving a fusion
 frame. Since, if in Proposition \ref{8},
 cosider $J=\{1,3\}$, so $a={4\over 3}<2=A.$

 $\mathcal {W}=\{(W_i,v_i)\}_{i=1}^{4}$ is a 2-tight fusion frame
  for $\CC^{5}$, and   $\mathcal {W}$ has uniform  redundancy,
$$ \mathcal{R}_{\mathcal{W}}^-= \mathcal{R}_{\mathcal{W}}^+=2.$$
\indent In this example, our definition of redundancy  coincides with
 the traditional concept of redundancy,  i.e.,  $\mathcal {W}$ is robust
  against  $2$-erasures.
\end{exam}

\textbf{Acknowledgments}
The Authors would like to thank the reviewers for their valuable comments and suggestions to improve the manuscript.

\bibliographystyle{amsplain}

\begin{thebibliography}{5}
\bibitem{1} R. Balan, P. G. Casazza and Z. Landau, {\it Redundacncy for localized frames,} Israel J. Math. \textbf{185} (2011), 445-476.
\bibitem{1'} J. Benedetto, A. Powell and O. Yilmaz, {\it Sigma-Delta quantization and finite frames,} IEEE Trans. Inform. Th. \textbf{52} (2006), 1990-2005.
\bibitem{2} B. G. Bodmann, P. G. Casazza and G. Kutyniok,
{\it A quantitative notion of redundancy for finite frames,} Appl.
Comput. Harmon. Anal. \textbf{30} (2011), 348-362.
\bibitem{3'} H. Bolcskei, F. Hlawatsch and H. G. Feichtinger, {\it Frames-theoretic analysis of oversampled filter banks,} IEEE Trans. Signal Processing,  \textbf{46} (1998), 3256-3268.
\bibitem{3}
A. M. Bruckstein, D. L. Donoho and M. Elad, {\it From sparse solutions of systems of equations to sparse modeling of signals and images,} SIAM Rev. \textbf{51} (2009) 34-81.
 \bibitem{cahill} J. Cahill, P.G. Casazza and A. Heinecke, {\it A notion of redundancy for infinite frames, } Proceedings of SAMPTA 2011, Singapore.
 \bibitem{5'} E. J. Cand\`{e}s and D. L. Donoho, {\it New tight frames of curvelets and optimal representations of objects with piecewise $C^2$ singularities,} Comm. Pure and Appl. Math. \textbf{56} (2004), 216-268.
\bibitem{5}
P. G.  Casazza and G. Kutyniok (Eds.), {\it Finite frames, theory and applications,} XVI, 485P. (2013).
\bibitem{7}
P. G. Casazza and G. Kutyniok, {\it Fusion Frames and Distributed Processing, }
Appl. Comput. Harmon. Anal. \textbf{25} (2008), 114-132.
\bibitem{8}
P. G. Casazza and G. Kutyniok, {\it Robustness of fusion frames under erasures of subspaces and of local frame vectors},  Radon transforms, Geometry, and Wavelets, 149-–160, Contemp. Math.  \textbf{464}, Amer. Math. Soc., Providence, RI, 2008.
\bibitem{6}
P. G. Casazza and G. Kutyniok, {\it Frames of subspaces,} in: Wavelets, Frames and Operator Theory (College Park, MD, 2003), Contempt. Math.  \textbf{345}, Amer. Math. Soc. Providence, RI. (2004), 87-113.
\bibitem{9'}
P. G. Casazza,  G. Kutyniok, S. Li and C. J. Rozell, {\it Modeling Sensor Networks with Fusion Frames,} Wavelets XII (San Diego, CA, 2007), SPIE Proc. \textbf{6701}, SPIE, Bellingham, WA, 2007.
\bibitem{10}
O. Christensen, {\it An introduction to frames and Riesz bases,}
Birkh\"{a}user, Boston 2003.
\bibitem{duffin}
R.J. Duffin, A.C. Schaeffer, {\it A class of nonharmonic Fourier
series,} {\it Trans. Amer. Math. Soc.} \textbf{72} (1952), 341-366.
\bibitem{Dykema}
K. Dykema, D. Freeman, K. Kornelson, D. R. Larson, M. Ordower and E. Weber, {\it Ellipsoidal tight frames and projection decompositions of operators,} Ill. J. Math. \textbf{48} (2004), 477-489.
\bibitem{for}
M. Fornasier, {\it Quasi-orthogonal decompositions of structured frames,} J. Math. Appl. \textbf{289} (2004), 180-199.
\bibitem{11}
P. G\v{a}vruta, {\it On the duality of fusion frames,} J. Math. Anal. Appl. \textbf{333} (2007), 871-879.
\bibitem{121} R. W. Heath and A. J. Paulraj, {\it Linear dispersion codes for MIMO systems based on frame theory,} IEEE Trans. Signal Processing \textbf{50} (2002), 2429-2441.
\bibitem{12'} S. S. Iyengar and R. R. Brooks, eds, {\it Distributed sensor networks,} Chapman  Hall/CRC, Baton Rouge, 2005.
\bibitem{12}
A. R. Mariano and D. Stojanoff, {\it Some properties of frames of subspaces obtained by operator theory methods,} J. Math. Anal. Appl. \textbf{343} (2008), 366-378.
\bibitem{parlet}
B. N. Parlett, {\it The symmetric eigenvalue problem,} 2nd ed., Classics in Applied Mathematics. vol. 20. Philadelphia: Society for Industrial and Applied Mathematics (SIAM), 1998.





\end{thebibliography}

%

\end{document}